\pdfoutput=1
\documentclass[11pt,a4paper]{article}

\usepackage{amssymb,amsmath,amsthm}
\usepackage[english]{babel}
\usepackage{t1enc}
\usepackage[latin2]{inputenc}
\usepackage{epsfig}
\usepackage{subfig}

\usepackage[portrait]{geometry}
\usepackage{comment}

\usepackage{hyperref}

\newtheorem{thm}{Theorem}
\newtheorem{corollary}[thm]{Corollary}
\newtheorem{lemma}[thm]{Lemma}
\newtheorem{proposition}[thm]{Proposition}

\newtheorem{remark}[thm]{Remark}

\newtheorem{defi}[thm]{Definition}

\usepackage{indentfirst}
\usepackage{xspace} 

\def\R{\mbox{\ensuremath{\mathbb R}}\xspace}

\def\C{\mbox{\ensuremath{\mathcal C}}\xspace}
\def\F{\mbox{\ensuremath{\mathcal F}}\xspace}
\def\HH{\mbox{\ensuremath{\mathcal H}}\xspace} 

\def\HABA{\mbox{\ensuremath{\mathcal{ABA\textit{-free}}}}\xspace}
\def\HABAB{\mbox{\ensuremath{\mathcal{ABAB\textit{-free}}}}\xspace}
\def\HABABA{\mbox{\ensuremath{\mathcal{ABABA\textit{-free}}}}\xspace}

\newcommand{\rectangle}{{%
		\ooalign{$\sqsubset\mkern3mu$\cr$\mkern3mu\sqsupset$\cr}%
	}}

\mathcode`l="8000
\begingroup
\makeatletter
\lccode`\~=`\l
\DeclareMathSymbol{\lsb@l}{\mathalpha}{letters}{`l}
\lowercase{\gdef~{\ifnum\the\mathgroup=\m@ne \ell \else \lsb@l \fi}}%
\endgroup

\begin{document}

\title{Coloring hypergraphs defined by stabbed pseudo-disks and $ABAB$-free hypergraphs\thanks{Research by the first author was partially supported by ERC AdG Disconv and MTA EU10/2016-11001. Research by the second author was supported by the National Research, Development and Innovation Office -- NKFIH under the grant K 116769. Research by the second and third authors was supported by the Lend\"ulet program of the Hungarian Academy of Sciences (MTA), under grant number LP2017-19/2017.}}
\author{Eyal Ackerman\\
\and Bal\'azs Keszegh\\
\and D\"om\"ot\"or P\'alv\"olgyi}


\maketitle

\begin{abstract}
	What is the minimum number of colors that always suffice to color every planar set of points such that any disk that contains enough points contains two points of different colors? 
	It is known that the answer to this question is either three or four.
	We show that three colors always suffice if the condition must be satisfied only by disks that contain a fixed point.
	Our result also holds, and is even tight, when instead of disks we consider their topological generalization, namely \emph{pseudo-disks}, with a non-empty intersection.
	Our solution uses the equivalence that a hypergraph can be realized by stabbed pseudo-disks if and only if
	it is \emph{$ABAB$-free}.
	These hypergraphs are defined in a purely abstract, combinatorial way and our proof that they are $3$-chromatic is also combinatorial.
\end{abstract}

\section{Introduction}
Given a family of hypergraphs $\HH$ and a positive integer $c$, let $m(\HH,c)$ denote the least integer such that the vertices of every hypergraph $H\in\HH$ can be colored with $c$ colors such that every hyperedge of size at least $m(\HH,c)$ is non-monochromatic (i.e., contains two vertices with different colors).
In other words, for every hypergraph $H\in\HH$ the sub-hypergraph of $H$ that consists of all the hyperedges of size at least $m(\HH,c)$ is \emph{$c$-colorable}.
We denote by $\chi_m(\HH)$ the least integer $c$ for which such a finite $m(\HH,c)$ exists.

A family of geometric (or topological) regions $\F$ and a set of points $S$ naturally define a hypergraph $H(S,\F)$ whose vertices are the points in $S$ and whose hyperedge set consists of every subset $S' \subseteq S$ for which there is a region $F'\in\F$ such that $S'=F'\cap S$.
The family of (finite) hypergraphs $\HH(\F)$ defined by a family of geometric regions $\F$ consists of all the hypergraphs $H(S,\F)$ for some (finite) point set $S$.
We also say that $\F$ can \emph{realize} $H(S,\F)$.
By a slight abuse of notation we thus write $m(\F,c)$ and $\chi_m(\F)$ instead of $m(\HH(\F),c)$ and $\chi_m(\HH(\F))$, respectively.

Typically, one is interested in determining whether it holds that $\chi_m(\F)=2$ or at least $\chi_m(\F) <\infty$ for a given family of geometric regions $\F$. 
These questions are motivated by problems concerning \emph{cover-decomposability} and \emph{conflict-free} colorings. For more about these connections we refer to the surveys \cite{surveycd, surveycf}.
For example, it is known~\cite{homotsquare} that $m(\F_\square,2) \le 215$, where $\F_\square$ is the family of axis-parallel squares in the plane. In other words, it is possible to color any set of points in the plane with the colors blue and red, such that every axis-parallel square that contains at least $215$ points from this set of points contains a blue point and a red point.
Since, by definition, $\chi_m(\HH) > 1$, it follows that $\chi_m(\F_\square)=2$.
On the other hand, considering the family of axis-parallel rectangles $\F_\rectangle$, it is known~\cite{CPST09} that $\chi_m(\F_\rectangle)$ is infinite.

An intriguing question is to determine $\chi_m(\F_\bigcirc)$, where $\F_\bigcirc$ is the family of disks in the plane. It follows from the Four Color Theorem and the planarity of Delaunay-triangulations, that any finite set of points can be $4$-colored such that no disk containing at least two points is monochromatic, i.e., $m(\F_\bigcirc,4)=2$, and thus $\chi_m(\F_\bigcirc) \le 4$.
It is also known~\cite{indec} that $\chi_m(\F_\bigcirc) > 2$. Moreover, $\chi_m(\F) > 2$
even when $\F$ is the family of \emph{unit} disks \cite{unsplittable}.
Therefore, it remains an open problem whether  $\chi_m(\F_\bigcirc) = 3$ or  $\chi_m(\F_\bigcirc) = 4$.

We consider a generalization of disks, namely, \emph{pseudo-disks}. Roughly speaking, a family of regions is a family of pseudo-disks if they behave like disks in the sense that the boundaries of every two regions intersect at most twice (see Definition \ref{def:psdisk} below for a more precise definition).
We say that a family of regions is \emph{stabbed} if their intersection is non-empty,
that is, there exist a point that \emph{stabs} (i.e., it is contained in) all the regions. We say that a family of regions is \emph{internally stabbed} if the intersection of their interiors is non-empty.
Our main result is that coloring with three colors is possible (and sometimes necessary) for families of stabbed pseudo-disks. 

\begin{thm}\label{thm:main}
	Let $\F$ be a family of pseudo-disks whose intersection is non-empty and let $S$ be a finite set of points. Then it is possible to color the points in $S$ with three colors such that any pseudo-disk in $\F$ that contains at least two points from $S$ contains two points of different colors.
	Moreover, for every integer $m$ there is a set of points $S$ and a family of pseudo-disks $\F$ with a non-empty intersection, such that for every $2$-coloring of the vertices of the hypergraph $H(S,\F)$ there is a hyperedge of size at least $m$ which is monochromatic.
	
	To summarize with our notation, $m(\F_\odot,3) = 2$ and $\chi_m(\F_\odot) = 3$, where we denote by $\F_\odot$ the families of stabbed pseudo-disks.
\end{thm} 

It is important to note that the above-mentioned construction from~\cite{unsplittable} of a family $\F$ of unit disks (or more generally, \emph{translates} of any region with a smooth boundary) such that $\chi_m(F) > 2$ is not a family of stabbed pseudo-disks (although it is stabbed by two points, that is, there are two points such that every region contains at least one of them).

From Theorem~\ref{thm:main} it is easy to conclude the following.
\begin{corollary}
	Given a finite set of points $S$ it is possible to color the points of $S$ with three colors such that any disk that contains the origin and at least two points from $S$ contains two points with different colors.
\end{corollary}

This corollary is already nontrivial for unit disks containing the origin. By a well-known duality concerning translates of regions (see e.g., \cite{surveycd}) we have:

\begin{corollary} It is possible to decompose a sufficiently thick covering of any region of radius at most one 
	by finitely many unit disks into three parts such that any two of the three parts cover the whole region.
\end{corollary}

We present two proofs for the upper bound $\chi_m(\F_\odot) \le 3$ of Theorem~\ref{thm:main}.
The first proof is a direct proof that uses some previous results about the so-called ``shrinkability'' of a family of pseudo-disks~\cite{buzaglo,pinchasi} that rely on a highly nontrivial sweeping machinery from~\cite{snoeyink}.\footnote{We would like to note that some of these papers (sometimes implicitly) assume stricter conditions, like no three pseudo-disks should pass through a point. We believe that these conditions could be removed with some extra care, but that would require to repeat the whole argument. Therefore, we do not go into details, especially since we also give a self-contained proof for our main result.}
For some of these results we provide new and simplified proofs.
Our second proof of the upper bound $\chi_m(\F_\odot) \le 3$ is completely self-contained and of a more combinatorial flavor.
It is based on an equivalence between hypergraphs defined by stabbed pseudo-disks and \emph{$ABAB$-free} hypergraphs. This equivalence also implies that $\chi_m(\F_\odot) \ge 3$ following a result from~\cite{abafree}.


\paragraph{$ABAB$-free hypergraphs.}
Let $l\geq 1$ be a number such that $2l$ is an integer.
We denote by $(AB)^l$ the alternating sequence of letters $A$ and $B$ of length $2l$. For example, $(AB)^{1.5}=ABA$ and $(AB)^2=ABAB$.

\begin{defi}[$(AB)^l$-free hypergraphs]
	\mbox{}
\begin{enumerate}
	\item Two subsets $A,B$ of an ordered set of elements form an \emph{$(AB)^l$-sequence} if there are $2l$ elements $a_1 < b_1 < a_2 < b_2 < \ldots$ such that $\{a_1,a_2,\ldots\} \subset A \setminus B$ and  $\{b_1,b_2,\ldots\} \subset B \setminus A$.
	\item A hypergraph with an ordered vertex set is {\em $(AB)^l$-free} if it does not contain two hyperedges $A$ and $B$ that form an $(AB)^l$-sequence.
	\item A hypergraph with an unordered vertex set is $(AB)^l$-free if there is an order of its vertices such that the hypergraph with this ordered vertex set is $(AB)^l$-free.
	\item The family of all $(AB)^l$-free hypergraphs is denoted by $\mathcal{(AB)}^l\emph{-free}$.
\end{enumerate}
\end{defi}

$(AB)^l$-free hypergraphs were introduced in~\cite{abafree}, where it was shown that $ABA$-free hypergraphs are equivalent to hypergraphs defined by pseudo-halfplanes.
It was also proved in~\cite{abafree} that $\chi_m(\HABA)=2$ (along with further strengthenings) and that $\chi_m(\HABAB)>2$.
\begin{thm}[\cite{abafree}]\label{thm:abab-no}
	For every $m \ge 2$ there exists an $ABAB$-free $m$-uniform hypergraph which is not $2$-colorable. 	
\end{thm}
Here we extend these results and show that $\chi_m(\HABAB)=3$ by proving that $m(\HABAB,3)=2$.

\begin{thm}\label{thm:abab}
	Every $ABAB$-free hypergraph is proper $3$-colorable.
\end{thm}

Theorem~\ref{thm:main} then follows from Theorems~\ref{thm:main} and~\ref{thm:abab-no} and an equivalence between $ABAB$-free hypergraphs and hypergraphs defined by stabbed pseudo-disks.

\begin{thm}\label{thm:equiv}
	A hypergraph is $ABAB$-free if and only if it can be realized by a family of stabbed pseudo-disks.
\end{thm}


As a side question we consider $(AB)^l$-free hypergraphs for $l > 2$ and show using a construction similar to the one from~\cite{abafree} that $\chi_m(\HABABA)=\infty$.

\begin{thm}\label{thm:ababa-free}
	For every $c\ge 2$ and $m\ge 2$ there exists an $ABABA$-free $m$-uniform hypergraph which is not $c$-colorable. 	
\end{thm}

\paragraph{Further related work.}
As we mentioned before, $\chi_m(\F_\square)=2$, where $\F_\square$ denotes the family of axis-parallel squares. By affine transformations the same result holds for families of homothets of a fixed parallelogram.
It is also known that $\chi_m(\F_\Delta)=2$, for each family $\F_\Delta$ of homothets of a given triangle~\cite{octant}. There are also good estimates of $m(\F_\Delta,2)$, namely, $5 \le m(\F_\Delta,2) \le 9$~\cite{oktans9}.
P\'alv\"olgyi and T\'oth~\cite{PT10} proved that for a family $\F$ of translates of a given open convex polygon $\chi_m(\F)=2$.
Perhaps the most interesting open problem concerning $2$-coloring is whether the same bound holds for homothets of a given convex polygon. P\'alv\"olgyi and Keszegh~\cite{3propercol} showed that $\chi_m \le 3$ in this case.
For further results about translates and homothets of convex shapes, see e.g., \cite{surveycd, octant, homotsquare, 3propercol} and the webpage~\cite{cogezoo}.

\paragraph{Outline.}
In Section~\ref{sec:ABAB-free} we prove that every $ABAB$-free hypergraph is $3$-colorable.
Then, in Section~\ref{sec:ABABA-free} we consider $ABABA$-free hypergraphs and prove that for every $c \ge 2$ there are such hypergraphs which are non-$c$-colorable.
In Section~\ref{sec:preliminaries} we recall some needed properties of pseudo-disks.
These properties are used in Section~\ref{sec:direct} to provide a direct proof of the upper bound in Theorem~\ref{thm:main}.
The equivalence between $ABAB$-free hypergraphs and hypergraphs defined by stabbed pseudo-disks is proved in Section~\ref{sec:parabolasgeom}.
We conclude with some remarks and open problems in Section~\ref{sec:conclusions}.

\section{Coloring $ABAB$-free hypergraphs}
\label{sec:ABAB-free}

In this section we prove Theorem \ref{thm:abab} which says
that every $ABAB$-free hypergraph is $3$-colorable.

Let $H$ be an $ABAB$-free hypergraph. A pair of vertices of $H$ is called {\em unsplittable} if by adding this pair as a hyperedge of size two to $H$ we get an $ABAB$-free hypergraph. For a pair of vertices $E=\{p,q\}$ we say that a hyperedge $B$ {\em splits} this pair if $E$ and $B$ form an $EBEB$- or $BEBE$-sequence.

\begin{lemma}
	\label{lem:unsplittable}
	Every hyperedge of an $ABAB$-free hypergraph contains a pair of vertices that is unsplittable.
\end{lemma}

\begin{proof}
	Let $A$ be a hyperedge of an $ABAB$-free hypergraph $H$.
	If $A$ is of size two, then its vertices form an unsplittable pair, for otherwise there would be a hyperedge $B$ that splits $A$ and this would contradict that $H$ is $ABAB$-free.
	
	Thus we may assume that $A$ is of size at least $3$. 
	Consider a left-to-right order of the vertices of $H$ by which $H$ is $ABAB$-free.
	We write $a<b$ if $a$ and $b$ are two vertices of $H$ such that $a$ is to the left of $b$.
	Denote the vertices of $A$ according to their order by $A=\{a_1,a_2,\dots a_k\}$. Two such vertices are called \emph{consecutive} if one follows the other in this order. We will prove that one of the consecutive pairs of vertices of $A$ is an unsplittable pair.
	
	Assume on the contrary that none of the consecutive pairs is unsplittable. A consecutive pair $E=\{a_i,a_{i+1}\}$ is \emph{left-splittable} (resp., \emph{right-splittable}) if there exists a hyperedge $B\in H$ such that they together form a $BEBE$-sequence (resp., $EBEB$-sequence). By our assumption every consecutive pair is either left-splittable or right-splittable or both. A consecutive pair is called \emph{one-sided} splittable (or simply \emph{one-sided}) if it is not both left-splittable and right-splittable. Notice that the leftmost consecutive pair $C=\{a_1,a_2\}$ cannot be left-splittable. Indeed, a hyperedge $B$ left-splitting it would also form a $BCBC$-sequence, which is a contradiction. Similarly, the rightmost consecutive pair cannot be right-splittable. Thus the family of one-sided splittable pairs is non-empty. 
	\begin{figure}
		\begin{center}
			\includegraphics{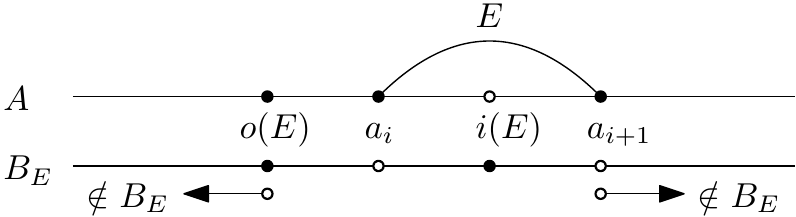}
			\caption{$E$ is an only left-sided pair with witness hyperedge $B_E$.}
			\label{fig:leftsidedpairdef}
		\end{center}
	\end{figure}
	
	For each only left-sided pair $E=\{a_i,a_{i+1}\}$ let $B_E$ be a hyperedge that together with $E$ forms a $B_EEB_EE$-sequence, see Figure \ref{fig:leftsidedpairdef}. The existence of this sequence implies that $a_i,a_{i+1}\in E\setminus B_E$ and that there is a vertex $i(E)\in B_E\setminus A$ among the vertices of $H$ between $a_i$ and $a_{i+1}$ (in the left-to-right order of the vertices of $H$). The leftmost vertex of $B_E$ is denoted by $o(E)$. As $E$ is left-sided and $B_E$ is a witness for that, it follows that $o(E) < a_i$. Also, $o(E) \in A \cap B_E$ since if $o(E) \notin A$ then $o(E),a_i,i(E),a_{i+1}$ would form a $B_EAB_EA$-sequence, a contradiction. Note that there is no vertex in $B_E$ to the left of $o(E)$ by definition and there is no vertex in $B_E$ to the right of $a_{i+1}$, for otherwise $E$ would also be a right-sided pair.
	
	Similarly, for each only right-sided pair $E=\{a_i,a_{i+1}\}$ take a witness hyperedge $B_E$ with which it forms an $EB_EEB_E$-sequence. Thus $a_i,a_{i+1}\in A\setminus B_E$ and there is a vertex $i(E)\in B_E\setminus A$ among the vertices of $H$ between $a_i$ and $a_{i+1}$. In this case denote by $o(E)$ the rightmost vertex of $B_E$. Therefore, $a_{i+1} < o(E)$ and, as before, we have that $o(E)\in A\cap B_E$.
	
	Among all one-sided pairs of $A$ let $E=\{a_i,a_{i+1}\}$ be the pair with the least number of vertices of $H$ between $i(E)$ and $o(E)$. Without loss of generality we may assume that $E$ is only right-sided.
	
	\begin{figure}
		\begin{center}
			\includegraphics{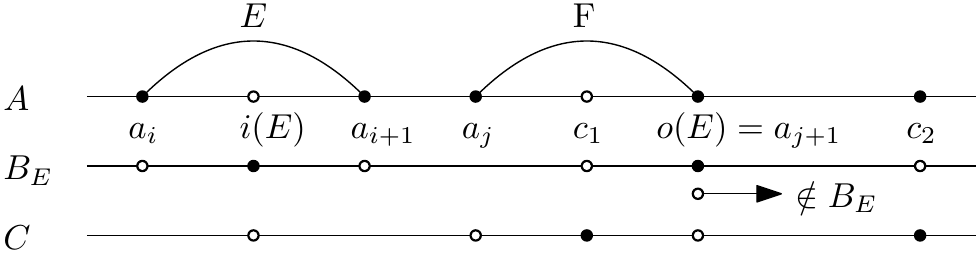}
			\caption{If $F$ is right-sided then $i(E),c_1,o(E),c_2$ form a $B_ECB_EC$-sequence.}
			\label{fig:Fnotrightsided}
		\end{center}
	\end{figure}

	As $o(E)\in A$ and $a_{i+1} < o(E)$, $o(E)=a_{j+1}$ for some $j>i$. Consider the pair $F=\{a_j,a_{j+1}\}$ (note that $a_j$ may coincide with $a_{i+1}$). We claim that $F$ cannot be a right-sided pair. Indeed, assume to the contrary that there exists a hyperedge $C$ and two vertices $c_1,c_2\in C\setminus F$ such that $a_j<c_1<a_{j+1}<c_2$ (and therefore $a_j,c_1,a_{j+1},c_2 $ form an $FCFC$-sequence), see Figure~\ref{fig:Fnotrightsided}. Since $o(E)=a_{j+1}<c_2$ and $o(E)$ is the rightmost vertex of $B_E$, we also have $c_2\notin B_E$. Also, $i(E)\notin C$, otherwise $i(E),a_j,c_1,a_{j+1}$ would form a $CACA$-sequence, a contradiction. Similarly, $c_1\notin B_E$, otherwise $a_i,i(E),a_{i+1},o(E)$ would form an $AB_EAB_E$-sequence. However, then $i(E),c_1, a_{j+1},c_2$ form a $B_ECB_EC$-sequence, which is again a contradiction.
	
	Therefore, $F$ is an only left-sided pair and thus $o(F)<a_j$. See Figure \ref{fig:ababfreeproof}.  Furthermore, $o(F)\le a_i$ for otherwise there would be less vertices of $H$ between $o(F)$ and $i(F)$ than there are between $o(E)$ and $i(E)$, contradicting our choice of $E$. We have that $o(F)\in B_F\cap A$ for otherwise $o(F),a_j,i(F),a_{j+1}$ would be a $B_FAB_FA$-sequence. Furthermore, $o(F)\notin B_E$ since $o(F) \le a_i$ and no vertex of $B_E$ is left of $a_i$. Similarly, $i(E)\ne B_F$ as otherwise $i(E),a_j,i(F),o(E)$ would form a $B_FAB_FA$-sequence. Finally, $i(F)\notin B_E$ for otherwise $a_i,i(E),a_{i+1},i(F)$ would form an $AB_EAB_E$-sequence.	
	
	\begin{figure}
		\begin{center}
			\includegraphics{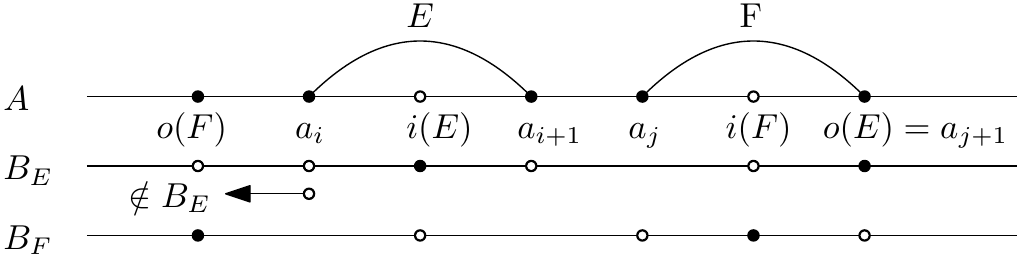}
			\caption{The vertices $o(F),i(E),i(F),o(E)$ form a $B_FB_EB_FB_E$-sequence.}
			\label{fig:ababfreeproof}
		\end{center}
	\end{figure}
	
	Thus, the vertices $o(F),i(E),i(F),o(E)$ form a $B_FB_EB_FB_E$-sequence, leading to the final contradiction. 
\end{proof}

\begin{proof}[Proof of Theorem \ref{thm:abab}]
	Let $H$ be an $ABAB$-free hypergraph.
	We call a hyperedge of size at least $3$ {\em unhit} if it does not contain as a subset a hyperedge of size $2$. 
	Starting from $H$ we create a series of hypergraphs as follows.
	If the current hypergraph contains an unhit hyperedge, then by Lemma~\ref{lem:unsplittable} this hyperedge contains an unsplittable pair which we add as a new hyperedge and obtain the next hypergraph in our series. 
	Since $H$ has a finite number of hyperedges and every hypergraph has one less unhit hyperedge than its preceding hypergraph, we get a finite series of hypergraphs. Let $H'$ be the last hypergraph in this series.
	
	Let $G$ be the graph that is induced by the hyperedges of $H'$ of size two. Note that every hyperedge of $H'$ contains at least one edge of $G$. Therefore, a proper coloring of $G$ is a proper coloring of $H$.
	The graph $G$ also has the $ABAB$-free property.
	Consider the following drawing of $G$. Its vertices are represented by distinct points on a horizontal line according to their $ABAB$-free order and its edges are drawn as circular arcs above the line.
	Since $G$ is $ABAB$-free its drawing does not contain crossing edges. Furthermore, this drawing of $G$ is outerplanar. Since every outerplanar graph is $3$-colorable, this completes the proof.
\end{proof}

As mentioned in the introduction, using Theorem \ref{thm:equiv} this also proves the upper bound of Theorem \ref{thm:main}.

\section{$ABABA$-free hypergraphs}
\label{sec:ABABA-free}

It is shown in~\cite{abafree} that there are $ABAB$-free hypergraphs that are not $2$-colorable. Here we extend this construction by proving  Theorem~\ref{thm:ababa-free} by showing that there are non-$c$-colorable $ABABA$-free hypergraphs for every $c \ge 2$.

We will use \emph{depth first search} (DFS)
to traverse the vertices (nodes) of a directed rooted tree.
The order in which a DFS search visits the vertices is called a DFS order.
The root of the tree is the first vertex that is visited in this search, and thus the first vertex in the DFS order.
In each subsequent step, we take the last (already visited) vertex in the DFS order that has a yet unvisited child, and visit one such child.

\begin{proof}[Proof of Theorem~\ref{thm:ababa-free}]
	Let $T(a,b)$ denote a full $a$-ary tree of depth $b-1$. That is, a tree in which every internal vertex has $a$ children and every leaf is at distance $b-1$ from the root of the tree (i.e., the path connecting the root to the leaf contains $b$ vertices). 
	
	Let the hypergraph $H(a,b)$ be defined as follows. Its vertex set is the vertex set of $T(a,b)$, the set of children of each internal vertex is a {\em horizontal} hyperedge of size $a$ and the set of vertices of every path from the root to a leaf is a {\em vertical} hyperedge of size $b$. 
	
	It is easy to see that in every coloring of the vertices of $H(a,b)$ with two colors there is either a monochromatic horizontal hyperedge (of size $a$) or a monochromatic vertical hyperedge (of size $b$). Therefore, $H_2 := H(m,m)$ is an $m$-uniform non-$2$-colorable hypergraph.	Let $n_2$ denote the number of vertices of $H_2$.
	
	For $c>2$ we define a non-$c$-colorable $m$-uniform hypergraph $H_c$ recursively.
	The vertices of $H_c$ are the vertices of $T(n_{c-1},m)$.
	The hyperedges of $H_c$ are defined as follows.
	For each set of $m$ vertices that lie on a path from the root of the tree to one of its leaves define a \emph{vertical} hyperedge; for each set of $n_{c-1}$ children of an internal vertex of $T(n_{c-1},m)$ define a hypergraph isomorphic to $H_{c-1}$ and add all of its hyperedges as \emph{horizontal} hyperedges of $H_c$.
	
	It follows from the definition that $H_c$ is an $m$-uniform hypergraph for every $c \ge 2$.
	It remains to show that $H_c$ is non-$c$-colorable and $ABABA$-free.
	
	\begin{proposition} \label{prop:non-c-colorable}
		$H_c$ is non-$c$-colorable for every $c \geq 2$.
	\end{proposition}
	
	\begin{proof}
		We prove by induction on $c$.
		We observed above that $H_2$ is non-$2$-colorable.
		Suppose that $c > 2$ and $H_{c-1}$ is non-$(c-1)$-colorable. 
		Assume by contradiction that $H_c$ is $c$-colorable and consider a proper coloring of its vertices by $c$ colors. 
		Recall that the vertices of $H_c$ are the vertices of $T(n_{c-1},m)$ and assume, without loss of generality, that the color of the root is red.
		Then, the color of one of its children must also be red, for otherwise, there is a $(c-1)$-coloring of the copy of $H_{c-1}$ induced by these children.
		Similarly, one of the children of that red child must also be colored red, etc., yielding a path of red vertices from the root to one of the leaves. However, this implies a monochromatic red hyperedge of $H_c$ which is a contradiction.
	\end{proof}
	
	\begin{proposition} \label{prop:ABABA-free}
		$H_c$ is $ABABA$-free.
	\end{proposition}
	
	\begin{proof}
		We prove by induction on $c$.
		First, we show that $H_2$ is $ABABA$-free,\footnote{We note that $H_2$ is in fact even $ABAB$-free~\cite{abafree}, but not when we take the vertices in DFS order. Instead we would need to take what we could call \emph{siblings first order} (this was defined only implicitly for $H_2$ in~\cite{abafree}). Quite surprisingly, however, if we use this ``siblings first order'' for $c>2$, we can have an `$ABABA$', i.e., this would only show that $H_c$ is $ABABAB$-free.} the general case will follow similarly. 		
		Recall that the vertices of $H_2$ are the vertices of $T(m,m)$. We claim that a DFS-order of these vertices is $ABABA$-free. 
		Indeed, let $A=\{a_1,a_2,\ldots,a_m\}$ and $B=\{b_1,b_2,\ldots,b_m\}$ be two hyperedges of $H_2$, such that the vertices of these hyperedges are listed in DFS order.
		Now we distinguish some cases.
		
		If both $A$ and $B$ are vertical, then their first few elements are equal, and after these the remaining elements of one precedes the other, i.e., their vertices are ordered as, say, $a_1=b_1, \ldots, a_i=b_i$, $a_{i+1}, \ldots, a_m$, $b_{i+1}, \ldots, b_m$ for some $i$.
		In this case they are even $ABA$-free.
		
		If both $A$ and $B$ are horizontal, then either $A=B$ or they are disjoint.
		In the latter case, their vertices are ordered as, say, $a_1, \ldots, a_i$, $b_1, \ldots, b_m$, $a_{i+1}, \ldots, a_m$ for some $i$.
		In this case they are $ABAB$-free.
		
		
		If $A$ is vertical and $B$ is horizontal, then they intersect in at most one element.
		Their vertices can be ordered as $a_1, \ldots, a_{i-1}$, $b_1, \ldots, b_j(=)a_i,$ $a_{i+1}, \ldots, a_m$, $b_{j+1} \ldots,b_m$ for some $i,j$, where $b_j$ and $a_i$ might be equal.
		In this case they are $BABA$-free, thus also $ABABA$-free.
		
		If $A$ is horizontal and $B$ is vertical, then the same argument gives that they are even $ABAB$-free.
		
		\smallskip
		
		The proof of the induction step is quite similar.
		Suppose that $c > 2$ and $H_{c-1}$ is $ABABA$-free.
		Recall that the vertices of $H_c$ are the vertices of the tree $T(n_{c-1},m)$.
		A vertical edge and another (vertical or horizontal) edge behave exactly the same way as for $H_2$, thus they will be $ABABA$-free using the same arguments.
		The only case left to check is if both $A$ and $B$ are horizontal, but then they are $ABABA$-free using induction, if during the DFS search we take the siblings in the order given by the induction.
	\end{proof}
	
	This concludes the proof of Theorem~\ref{thm:ababa-free}.		
\end{proof}

\begin{corollary}
	For every triple of integers $l\ge 2.5$, $c \ge 2$ and $m \ge 2$ there exists an $(AB)^l$-free $m$-uniform hypergraph which is not $c$-colorable. 	
\end{corollary}

\section{Properties of pseudo-disks}\label{sec:preliminaries}

In this section we present and in some cases also (re)prove some properties of (stabbed) pseudo-disks. 
We begin by formally defining pseudo-disks.

We always assume that the boundaries of two regions intersect in finite many points. We say that an intersection point  $p$ of two regions is a \emph{touching point}, if $p$ lies on their boundaries but there is a neighborhood of $p$ such that there is an arbitrarily small perturbation of the boundaries that makes them disjoint inside this neighborhood of $p$. If an intersection point of two boundaries is not a touching point, then we say that it is a point where they \emph{properly cross}.

\begin{defi}\label{def:psdisk}
	A family of compact regions in the plane, each of which is a region bounded by a Jordan curve\footnote{Note that by the Jordan--Schoenflies theorem such regions are always simply connected.}, is called a family of \emph{pseudo-disks} if any two regions are either disjoint, intersect exactly once in a touching point, or their boundaries intersect exactly twice, both times properly crossing. 
\end{defi}

We now recall some properties of pseudo-disks.

\begin{lemma}[\cite{buzaglo}]\label{lem:disjointedges}
	Let $D_1$ and $D_2$ be two pseudo-disks in the plane. Let $x$ and $y$ be two
	points in $D_1\setminus D_2$. Let $a$ and $b$ be two points in $D_2\setminus D_1$. Let $e$ be any Jordan
	arc connecting $x$ and $y$ that is fully contained in $D_1$. Let $f$ be any Jordan arc connecting $a$ and $b$
	that is fully contained in $D_2$. Then $e$ and $f$ cross an even number of times.
\end{lemma}

\begin{lemma}[\cite{pinchasi}]\label{lem:shrink2}
	Let $\F$ be a finite pseudo-disk family $\F$, let $S$ be a finite set of points and let $F\in \F$ be a pseudo-disk that contains exactly $k$ points of $S$, one of which is $q\in S$. Then for every integer $2\le l\le k$ there exists a set $F'\subset F$ such that $q\in F'$, $|F'\cap S|=l$ and $\F\cup\{F'\}$ is a family of pseudo-disks.
\end{lemma}


We will need to use the fact that the intersection and union of stabbed pseudo-disks are \emph{simply connected} (see also \cite{lenses} for such a statement).

\begin{thm}\label{thm:psdisksimplyconn}
	The union and the intersection of every finite family of stabbed pseudo-disks are both simply connected.
\end{thm}

As in the literature there seems to be a bit of confusion (and even false claims) as  for why Theorem~\ref{thm:psdisksimplyconn} holds, we provide a new and simple proof for it (in a more general form) in Appendix \ref{appendix}.
The following lemma will be used to prove the equivalence between hypergraphs defined by stabbed pseudo-disks and $ABAB$-free hypergraphs.

\begin{lemma}\label{lem:ray}
	Let $\F$ be a finite family of pseudo-disks each of which contains a common point $p$ in their interior. Then for any point $q$ there is a Jordan arc from $p$ to infinity that contains $q$ and intersects the boundary of every pseudo-disk in $\F$ exactly once. We can further require that this arc intersects the boundaries of the pseudo-disks in $\F$ in distinct points (possibly except for $q$).
\end{lemma}

We first need to define the \emph{arrangement} determined by the boundaries of a finite family of Jordan regions whose boundaries intersect finitely many times (in particular, of a finite pseudo-disk family). The \emph{vertices} of the arrangement are the intersection points of the boundaries of the pseudo-disks, the \emph{edges} are the maximal connected parts of the boundaries that do not contain a vertex and the \emph{faces} are the maximal connected parts of the plane which are disjoint from the edges and the vertices of the arrangement. Thus the faces are maximal regions which are either completely contained in or disjoint from every region. The Jordan--Schoenflies theorem implies that all faces except the outer face are homeomorphic to a disk $D$ while the outer face is homeomorphic to $\R^2\setminus D$.

\begin{proof}[Proof of Lemma \ref{lem:ray}]
	Consider the arrangement defined by $\F$. By Theorem \ref{thm:psdisksimplyconn} the intersection of all pseudo-disks is simply connected and thus it must be exactly the face $f_p$ of the arrangement that contains $p$. Also by Theorem \ref{thm:psdisksimplyconn} the union of all the pseudo-disks is simply connected and thus in the arrangement there is exactly one face $f_\infty$ disjoint from all pseudo-disks.
	
	We claim that every face $f$ different from $f_p$ is adjacent to an edge which is on the boundary of a pseudo-disk that does not contain $f$. Indeed, if all the pseudo-disks that share some boundary with $f$ would contain $f$, then their intersection would contain $f$.
	However, their intersection also contains $f_p \ne f$, which is a contradiction, since the intersection of these pseudo-disks is simply connected by Theorem~\ref{thm:psdisksimplyconn}.  
	
	
	Using similar arguments (see~\cite[Corollary~2.10]{lenses}), we can also conclude that every face $f$ except $f_\infty$ contains a boundary part which is shared by the boundary part of a pseudo-disk which contains $f$.\footnote{One can also apply the projection trick from the proof of Lemma \ref{lem:molnarintersection} that can be found later.}.	
	
	For a face $f$ define its \emph{depth} as the number of pseudo-disks that contain $f$. Thus it follows from the arguments above that: $f_p$ is the only face with maximal depth $|\F|$; $f_\infty$ is the only face with depth $0$; every face with depth $d>0$ has a neighboring face whose depth is $d-1$; and every face with depth $d<|\F|$ has a neighboring face whose depth is $d+1$. 
	
	The \emph{dual graph} of the arrangement of pseudo-disks has the faces of the arrangement as its vertex set and contains an edge for every two neighboring faces (that is, faces whose boundaries share a common edge of the arrangement).
	It follows that the distance of any face $f$ from $f_\infty$ in the dual graph is exactly its depth and, similarly, its distance from $f_p$ is the depth of $f_p$ minus the depth of $f$.
	
	Assume first that $q$ is inside some face $f_q$.
	It is possible to draw a curve that connects $q$ to $p$ through the faces that correspond to the shortest path between $f_q$ and $f_p$ in the dual graph such that for every pair of consecutive faces along this path their common boundary is intersected exactly once at a point that is an interior point of the corresponding edge (of the arrangement).
	Similarly, it is possible to draw a curve that connects $q$ and a point in $f_\infty$ through faces that correspond to the shortest path between $f_q$ and $f_\infty$ in the dual graph.
	It is also possible to ensure that these curves intersect only at $q$.
	Let $\gamma$ denote the union of these two curves.
	
	As the faces along $\gamma$ have strictly decreasing/increasing depth, all these intersected boundaries belong to pairwise different pseudo-disks. Clearly $\gamma$ must also intersect all pseudo-disks (exactly once) as $p$ is inside all of them while $\infty$ is outside all of them.
	
	The additional property that the curve intersects the boundaries in different points follows from the construction as we always cross the inside of boundary parts between two neighboring faces.
	
	
	Suppose now that $q$ is on the boundary of one or more pseudo-disks. Note that $q$ cannot be a touching point of two pseudo-disks since otherwise these pseudo-disks will not have a common interior point.  Let $q'$ be a point inside all of these pseudo-disks arbitrarily close to $q$ and let $q''$ be a point outside all these pseudo-disks arbitrarily close to $q$. By Theorem \ref{thm:psdisksimplyconn} these pseudo-disks intersect in a simply connected region which contains $q$ and thus one face of the arrangement incident to $q$ is in all of these pseudo-disks, implying the existence of $q'$. The face opposite to this face in the circular order around $q$ is necessarily outside all of these pseudo-disks, using the fact that no two pseudo-disks touch each other at $q$. This implies the existence of $q''$.
	
	Connect $p$ to $q'$ and $q''$ to $\infty$ the same way as we did above and finally connect $q'$ to $q''$ with a curve through $q$ whose only intersection with boundaries is $q$. These curves together form the required curve.
\end{proof}



\section{A direct proof of the upper bound in Theorem~\ref{thm:main}}
\label{sec:direct}

Recall that Theorem~\ref{thm:main} follows from Theorems~\ref{thm:abab}, \ref{thm:abab-no} and~\ref{thm:equiv}.
Using the properties of pseudo-disks from the previous section, we can give an alternative and direct proof for the upper bound in Theorem~\ref{thm:main}.


\begin{proof}[Proof of the upper bound in Theorem~\ref{thm:main}]
	Let $\F$ be a family of pseudo-disks whose intersection is non-empty and let $S$ be a finite set of points.
	We wish to show that it is possible to color the points in $S$ with three colors such that any pseudo-disk in $\F$ that contains at least two points from $S$ contains two points of different colors.
	
	Consider a finite subfamily $\F' \subset \F$ that defines the same hypergraph on $S$ as $\F$. By applying Lemma~\ref{lem:shrink2} for every pseudo-disk (with $l=2$ and an arbitrary point inside the pseudo-disk) we extend $\F'$ such that every pseudo-disk with at least two points from $S$ contains a pseudo-disk with exactly two points from $S$. The pairs of points for which there is a pseudo-disk containing exactly these two points form the edges of the so-called \emph{Delaunay-graph} (of $S$ with respect to $\F$). It follows that by properly $2$-coloring this Delaunay-graph one obtains a proper $2$-coloring of the hypergraph $H(S,\F)$.
	
	We draw each edge of the Delaunay-graph $G$ such that it lies in one pseudo-disk that contains its two endpoints. This defines a drawing of $G$ in which edges may intersect, however, by Lemma~\ref{lem:disjointedges} independent edges intersect an even number of times.
	
	Consider the subdivision of the plane into faces that the drawing of $G$ defines.
	We claim that every point of $S$ is incident to the unbounded face. Indeed, otherwise there is a cycle in $G$ whose drawing separates a point $q\in S$ from infinity. However, then the union of the pseudo-disks corresponding to these edges is not simply connected (as it separates $q$ from infinity), contradicting Theorem \ref{thm:psdisksimplyconn}. 
	
	This implies that $G$ is an outerplanar graph (note that the embedding of $G$ that we consider is not necessarily a plane embedding).
	Indeed, connect all the points in $S$ to a new point $p'$ in the unbounded face such that the new edges do not cross each other and the original edges. Denote the resulting graph by $G'$ and note that we get a drawing of $G'$ such that independent edges cross an even number of times. Therefore by the Hanani-Tutte Theorem $G'$ is a planar graph. Consider a plane embedding of $G'$ and delete $p'$ from this embedding. We obtain a plane embedding of $G$ such that one face is incident to all the vertices. Therefore $G$ is outerplanar. Since outerplanar graphs are $3$-colorable, this completes the proof.
\end{proof}

%
%
%

\section{Proof of Theorem~\ref{thm:equiv}}
\label{sec:parabolasgeom}

In this section we prove Theorem \ref{thm:equiv}, that is, the equivalence between hypergraph families defined by stabbed pseudo-disks and $ABAB$-free hypergraphs. We prove this equivalence\footnote{Two hypergraph families are equivalent if for every hypergraph in one family there is an isomorphic hypergraph in the other family.} in several steps. First, we prove that stabbed pseudo-disks and internally stabbed pseudo-disks define the same hypergraphs. Then we prove that internally stabbed pseudo-disks and \emph{pseudo-parabolas} define the same hypergraphs.
Recall that a family of $x$-monotone bi-infinite Jordan curves\footnote{Let us call an $x$-monotone bi-infinite curve an `$x$-monotone bi-infinite Jordan curve' if for every $x_1<x_2$ the points of the curve which have $x$-coordinate between $x_1$ and $x_2$ form a Jordan arc.} is a family of pseudo-parabolas if every two curves in the family intersect at most twice.
A set of points $S$ and a (finite) family of pseudo-parabolas $\C$ naturally define a hypergraph $H(S,\C)$ whose vertex set is $S$ and whose hyperedge set consists of every subset $S' \subset S$ that is exactly the points from $S$ that lie on or above some pseudo-parabola in $\C$.
The equivalence between hypergraphs defined by internally stabbed pseudo-disks and hypergraphs defined pseudo-parabolas was already mentioned in~\cite{lenses} relying on a result of Snoeyink and Hershberger~\cite{snoeyink} by which a family of stabbed pseudo-disks can be swept by a ray. Here we reprove this equivalence using more elementary tools. Finally, we prove that hypergraphs defined by pseudo-parabolas are exactly the $ABAB$-free hypergraphs.

\subsection{Stabbed pseudo-disks and internally stabbed pseudo-disks}

\begin{proposition}\label{prop:perturb}
	Given any finite family $\F$ of stabbed pseudo-disks and a finite point set $S$, $\F$ can be perturbed to obtain an internally stabbed family $\F'$ such that $H(S,\F)= H(S,\F')$.
\end{proposition}

\begin{proof}
	Assume that the intersection of the pseudo-disks in $\F$ is a single point $p$ (as otherwise they are necessarily internally stabbed). Let $\F_p \subseteq \F$ be the pseudo-disks that contain $p$ on their boundaries. Take a disk $D$ centered at $p$ that is so small such that it does not contain points from $S$ (except $p$ if $p \in S$) and does not intersect the boundary of any pseudo-disk from $\F \setminus \F_p$.
	
	For each pseudo-disk $D_i \in \F_p$ let $x_i$ and $y_i$ be the first intersection points of $\partial D_i$ and $\partial D$ as we follow $\partial D_i$ from $p$ in either directions.\footnote{Note that it is possible that $\partial D$ intersects $\partial D_i$ an infinite number of times. However, since both of them are compact so is their intersection. In this case $x_i$ is the limit point of the intersection and is contained in it.} 
	Denote by $\gamma_i$ the part of $\partial D_i$ between $x_i$ and $y_i$ that does not contain $p$, and let $\varepsilon_i$ be a positive number such that every point on $\gamma_i$ is at distance greater than $\varepsilon_i$ from $p$ (due to compactness such a number exists). Let $\varepsilon$ be the minimum taken over all the $\varepsilon_i$'s and let $E$ be a disk of radius $\varepsilon$ centered at $p$.
	
	Now for every pseudo-disk $D_i \in \F_p$ let $x'_i$ (resp., $y'_i$) be the first intersection point with $\partial E$ on the segment of $\partial D_i$ from $x_i$ (resp., $y_i$) to $p$.\footnote{As before, such a point exists due to compactness.}
	Replace the segment of $\partial D_i$ between $x'_i$ and $y'_i$ with a radial segment, a part of a circle centered at $p$ (and inside $E$) and another radial segment, such that the resulting region contains $p$. See Figure~\ref{fig:2internally} for an illustration.
	\begin{figure}
		\centering
		\subfloat[]{\includegraphics[width= 6cm]{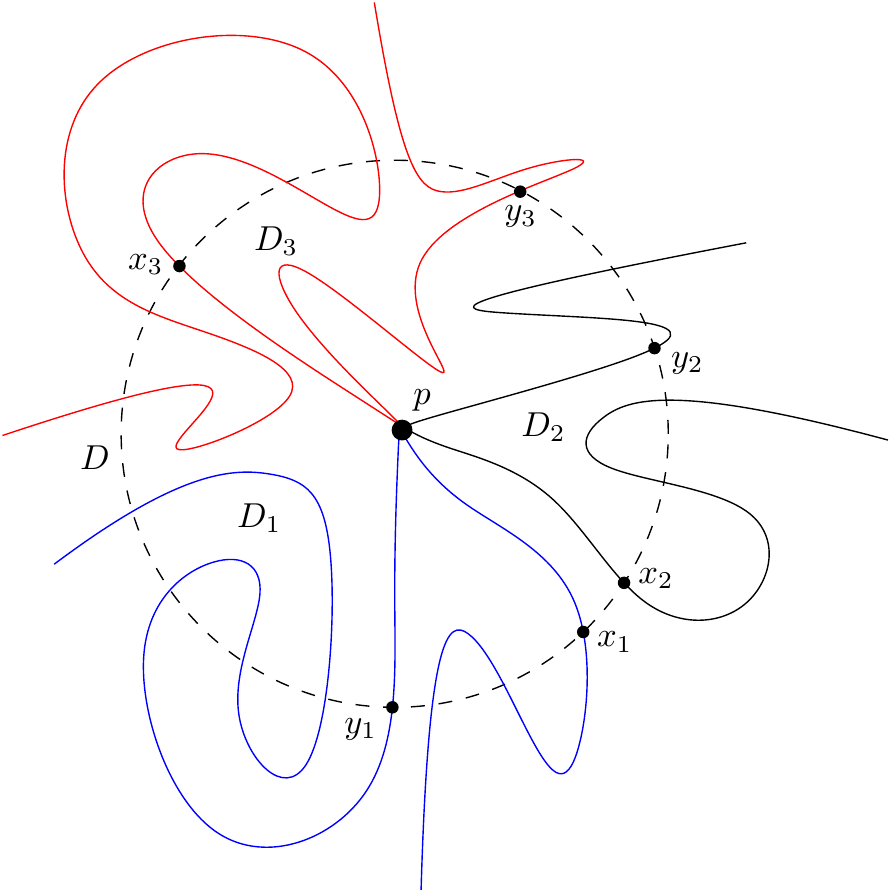}}
		\hspace{5mm}
		\subfloat[]{\includegraphics[width= 6cm]{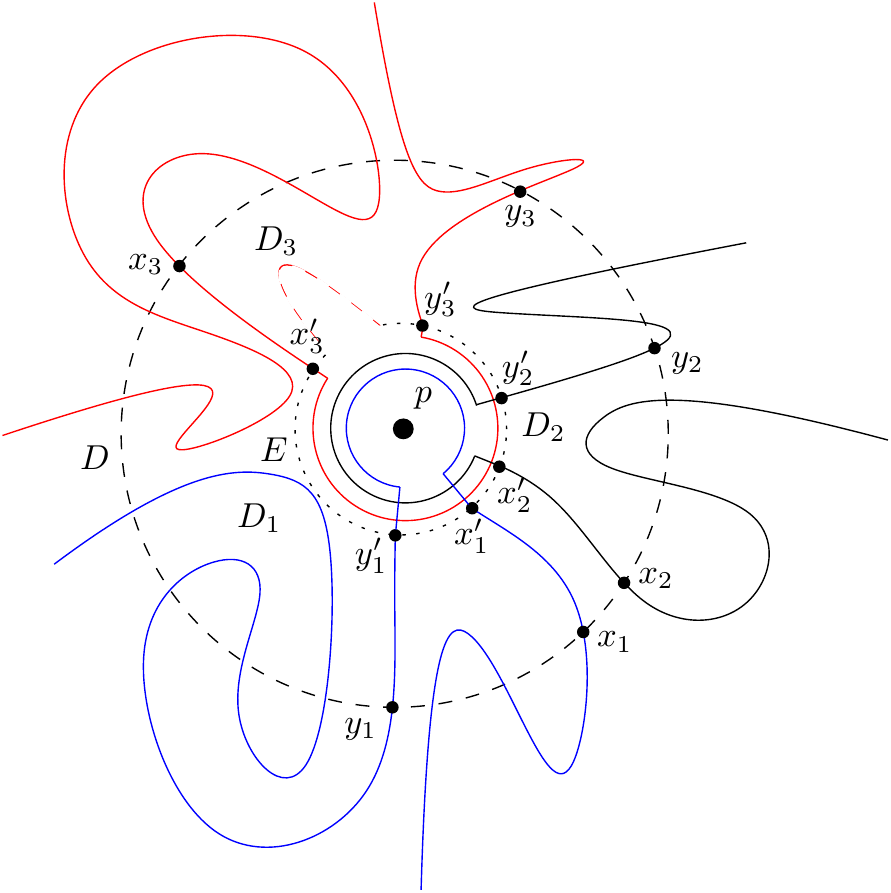}}
		\caption{Turning stabbed pseudo-disks into internally stabbed pseudo-disks.}
		\label{fig:2internally}		
	\end{figure}
	It is easy to see that the resulting family $\F'$ is a pseudo-disk family such that $H(S,\F)= H(S,\F')$.
\end{proof}

\subsection{Internally stabbed pseudo-disks and pseudo-parabolas}

In this section we prove that hypergraphs defined by internally stabbed pseudo-disks are equivalent to hypergraphs defined by pseudo-parabolas.
Following \cite{lenses}, we say that the combinatorial structure of an arrangement is its face lattice together with the containment relations between the regions and the faces. We call two arrangements combinatorially equivalent if the face lattices of their arrangements are isomorphic. 

Let $\C$ be a family of pseudo-parabolas and let $S$ be a finite set of points.
Recall that the hypergraph $H(S,\C)$ is the hypergraph whose vertex set is $S$ and whose hyperedge set consists of every subset $S' \subset S$ that is exactly the points from $S$ that lie on or above some pseudo-parabola in $\C$. 
We say that $\C$ is \emph{even} if every pair of pseudo-parabolas in $\C$ is either disjoint or intersects exactly twice.
It is easy to see that $\C$ is even if and only if the vertical orders of its members in $-\infty$ and $+\infty$ are the same.

\begin{proposition}\label{prop:even}
	Let $\C$ be a family of pseudo-parabolas and let $S$ be a set of points.
	Then there exists a family of pseudo-parabolas $\C'$ such that $H(S,\C)$ and $H(S,\C')$ are isomorphic and $\C'$ is even.
\end{proposition} 

\begin{proof}
	Let $x=M$ be a vertical line to the right of all the intersection points in the arrangement of $\C$ and the points of $S$. For each pseudo-parabola $C \in \C$ do the following. Suppose that $C$ is the $i$th pseudo-parabola in the bottom to top order of the pseudo-parabolas in $\C$ at $-\infty$. Replace the part of $C$ to the right of $x=M$ with a straight-line segment between the intersection of $x=M$ and $C$ and the point $(M+1,i)$ and a horizontal ray to $+\infty$ that starts at the latter point.
	It is easy to see that this results in an even family of pseudo-parabolas $\C'$ such that $H(S,\C)=H(S,\C')$.
\end{proof}

Given an even family of $n$ pseudo-parabolas $\C$, we can obtain a family of stabbed pseudo-disks as follows. Let $M$ be a large positive integer such that all the intersection points of pseudo-parabolas in $\C$ are between the lines $x=-M$ and $x=M$ and every pseudo-parabola in this range lies below the line $y=M$.
For each pseudo-parabola $C \in \C$ do the following. Suppose that $C$ is the $i$th pseudo-parabola in the bottom to top order of the pseudo-parabolas in $\C$ at $-\infty$ and $+\infty$.
Replace the parts of $C$ left of $x=-M$ and right of $x=M$ with a polygonal chain that connects the following points in this order: the intersection of $C$ and $x=-M$; the point $(-M-n+i,M)$; the point $(0,M+n-i)$; the point $(M+n-i,M)$; and the intersection of $C$ and $x=M$.

We call this arrangement of stabbed pseudo-disks the \emph{compactification} of $\C$. This new arrangement is combinatorially equivalent to the arrangement of $\C$ except that pairs of faces going to $-\infty$ and $+\infty$ between the same pseudo-parabolas are merged. 

Now we show that the converse also holds.

\begin{thm}\label{thm:equi}
	Compactifications of even pseudo-parabolas are combinatorially equivalent to internally stabbed pseudo-disks. That is, for every finite stabbed pseudo-disk arrangement $\mathcal D$ there is an arrangement of even pseudo-parabolas whose compactification gives an arrangement combinatorially equivalent to the arrangement $\mathcal D$ and vice versa.
\end{thm}

In \cite{lenses} it is proved that for any finite family of pseudo-disks in general position (that is, no three of their boundaries intersect in a common point) all containing a point $p$, there exists a combinatorially equivalent family of pseudo-disks, all of which are star-shaped with respect to $p$\footnote{A region is star-shaped with respect to $p$ if every line through $p$ intersects the region in a segment containing $p$.}. This easily implies Theorem~\ref{thm:equi} for stabbed pseudo-disks in general position. 
In order to prove Theorem~\ref{thm:equi} for stabbed pseudo-disks that are not in general position one can track the proof~\cite{lenses} (it seems that the general position assumption is not crucial),
however, we choose to provide a different proof based on Lemma~\ref{lem:ray}.

\begin{proof}[Proof of Theorem \ref{thm:equi}]
	As we mentioned above, the compactification of even pseudo-parabolas yields a combinatorially equivalent arrangement of internally stabbed pseudo-disks.
	
	For the other direction, consider an arrangement of pseudo-disks all containing a point $o$ in their interiors. Let $q_0$ be a point in the outer face of the arrangement, and also ``fix'' an infinitely far vertex $\infty$. 
	
	Denote the (finitely many) vertices of the arrangement of the pseudo-disks (that is, the intersection points of the boundaries of the pseudo-disks) by $q_2,q_3,\dots q_k$. Apply Lemma~\ref{lem:ray} to $q_0$ to get a curve $\gamma_0$ connecting $o$ with $\infty$ through $q_0$.
	For technical reasons, take also a curve $\gamma_{1}$ very close to $\gamma_1$ (with no intersection points of boundaries between them) and let $q_{1}$ be a point on it very close to $q_0$.
	
	In a general step, we have internally disjoint curves $\gamma_j$ for $0\le j\le i-1$ connecting $o$ to $\infty$ through $q_j$, 
	and we wish to create a curve $\gamma_i$ using Lemma \ref{lem:ray} connecting $o$ to $\infty$ through $q_i$. However, this curve may intersect previous curves, which we can rectify in several ways; here we present one possible solution.
		
	The previous curves slice up the plane around $o$ into ``pie-slices''. 
	Consider the slice containing $q_i$ and close it far away (outside a disk containing all pseudo-disks) to form a bounded region containing $q_i$ and the virtual vertex $\infty$.
	Define $Q_i$ as the union of this bounded region with a small enough disk around $o$; this disk should be so small that it is inside every pseudo-disk (such a small disk exists as all pseudo-disks are homeomorphic to a disk and their interiors contain $o$).
	By the properties of the curves, if we add $Q_i$ to the pseudo-disk family, we will still have a stabbed pseudo-disk family. Apply Lemma~\ref{lem:ray} to this family to get a curve $\gamma_i$ from $o$ to $\infty$ through $q_i$. Since $\gamma_i$ does not intersect the boundary of $Q_i$ 
	it cannot intersect any of the previous curves $\gamma_j$, $j < i$ (see Figure~\ref{fig:disks2parabolas1} for an example).
\begin{figure}
	\centering
	\subfloat[Adding the curve $\gamma_7$]{\label{fig:disks2parabolas1}
		{\includegraphics[width=6cm]{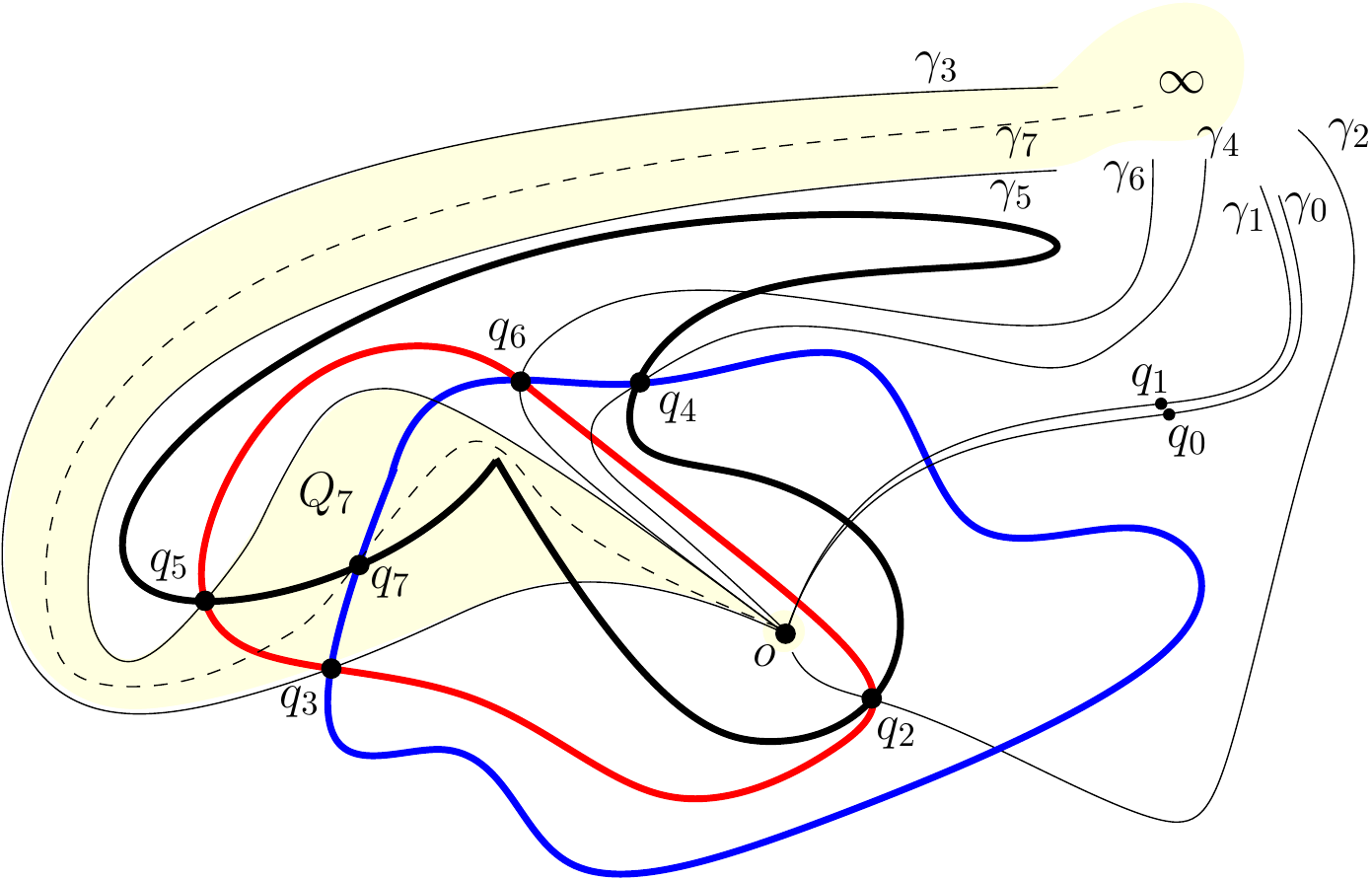}}}
	\hspace{5mm}
	\subfloat[Relabeling the curves]{\label{fig:disks2parabolas2}
		{\includegraphics[width=6cm]{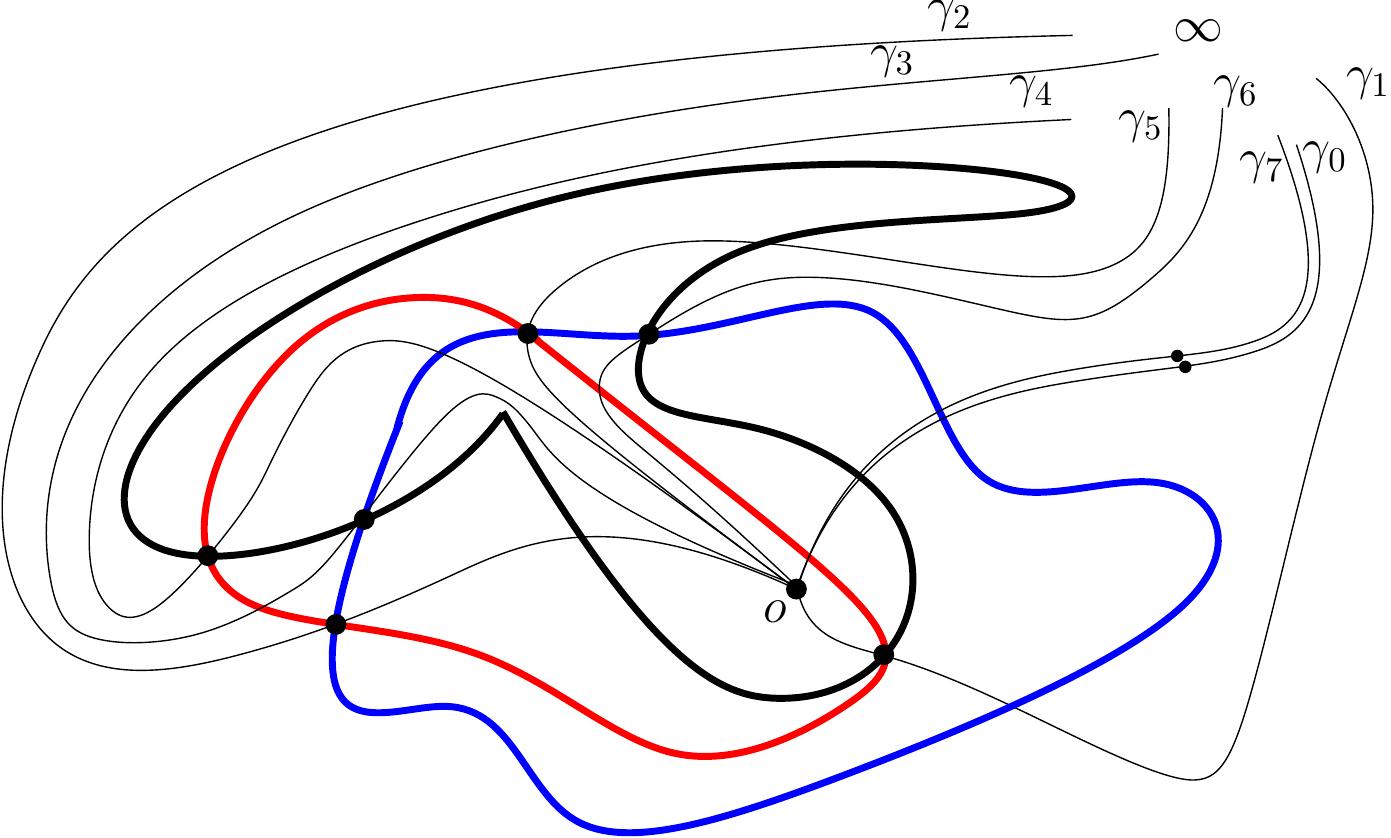}}}
	\hspace{5mm}
	\subfloat[The pseudo-parabolas arrangement]{\label{fig:disks2parabolas3}
		{\includegraphics[width=6cm]{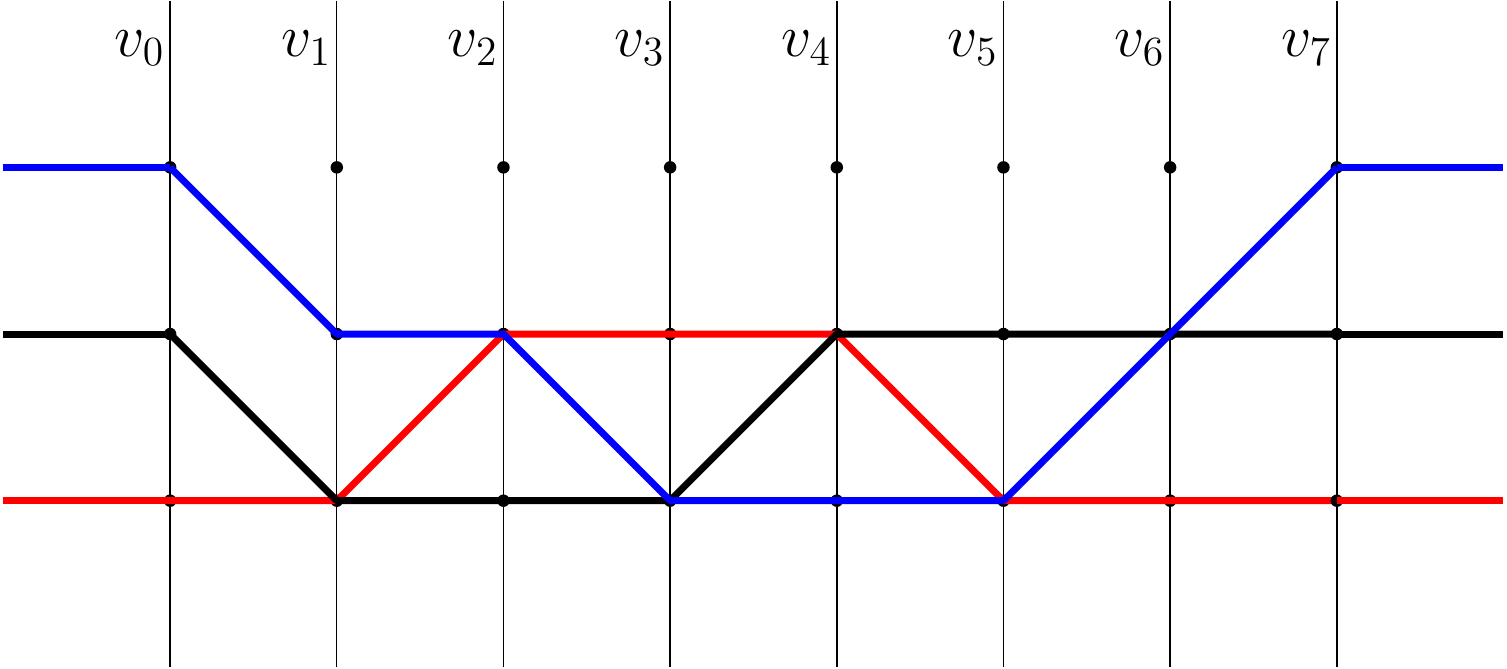}}}
	\caption{Illustrations for the proof of Theorem \ref{thm:equi}.}
	\label{fig:disks2parabolas}
\end{figure}	
	
	For convenience, we relabel the the curves according to their circular order around $o$, such that $\gamma_0$ remains the same curve and $\gamma_1$ becomes the last curve $\gamma_k$ (see Figure~\ref{fig:disks2parabolas2} for an example).
	Let $D_i$ be a pseudo-disk. Note that $\partial D_i$ intersect each of the curves exactly once and the interior of the segment of $\partial D_i$ between its intersection with two consecutive curves is crossing-free.
	Fix $k+1$ distinct vertical lines $v_0,v_1,\ldots,v_k$ ordered from left to right, and on each line $v_i$ designate $n_i$ points, where $n_i$ is the number of intersection points on $\gamma_i$, $i=0,\ldots,k$.
	Suppose also that the intersection points on every curve $\gamma_i$ are ordered from $o$ to $\infty$.
	Let $D_j$ be a pseudo-disk. Then we draw a corresponding polygonal curve (pseudo-parabola) $C_j$ as follows. For every $i=0,\ldots,k-1$ we connect by a straight-line segment the $a$th bottom-to-top designated point on $v_i$ to the $b$th bottom-to-top designated point on $v_{i+1}$, if $D_j$ intersects $\gamma_i$ (resp., $\gamma_{i+1}$) at the $a$th (resp., $b$th) intersection point along $\gamma_i$ (resp., $\gamma_{i+1}$).
	To make $C_j$ bi-infinite, we add horizontal rays with apexes at its left and right endpoints. See Figure~\ref{fig:disks2parabolas3} for an example.
	
	It is not hard to see that the resulting set of polygonal curves are pseudo-parabolas whose compactification (or compactification after reflecting about the $x$-axis) is combinatorially equivalent to the arrangement of pseudo-disks.
	This completes the proof. 
\end{proof}

\subsection{Pseudo-parabolas and $ABAB$-free hypergraphs}
\label{sec:parabolascomb}

A family of curves in the plane is called \emph{$t$-intersecting} if every two curves in the family intersect in at most $t$ points. Let $\C$ be a family of $x$-monotone bi-infinite curves in the plane and let $S$ be a finite set of points. Then the hypergraph $H(S,\C)$ has $S$ as its vertex set and contains a hyperedge $S'$ for every subset $S' \subseteq S$ for which there is a curve $c \in \C$ such that $S'$ consists of the points in $S$ that lie on or above $c$.
The following lemmas establish a generalization of the equivalence of $ABAB$-free hypergraphs and hypergraphs that are defined by pseudo-parabolas.

\begin{lemma}\label{lem:parabolas2ABAB}
	Let $\C$ be a family of $t$-intersecting $x$-monotone bi-infinite curves in the plane and let $S$ be a finite set of points. Then $H(S,\C)$ is an $(AB)^{(t+2)/2}$-free hypergraph.
\end{lemma}

\begin{proof}
	Order the points in $S$ from left to right. If two points have the same $x$-coordinate, then the point with the smaller $y$-coordinate precedes the other point. We claim that under this order $H(S,\C)$ is $(AB)^{(t+2)/2}$-free.
	Indeed, let $A$ and $B$ be two hyperedges of $H(S,\C)$ and let $c_A$ and $c_B$ be their corresponding curves.
	An $AB$-sequence implies that there are points $p_1,p_2 \in S$ with $p_1$ to the left of $p_2$ such that $p_1$ is on or above $c_A$ and below $c_B$ and $p_2$ is on or above $c_B$ and below $c_A$. Therefore $c_A$ and $c_B$ must intersect at some point between $p_1$ and $p_2$.
	It follows that an $(AB)^{(t+2)/2}$-sequence (whose length is $t+2$) implies that there are at least $t+1$ intersection points between $c_A$ and $c_B$, which is impossible since $\C$ is $t$-intersecting.
\end{proof}

\begin{lemma}\label{lem:ABAB2parabolas}
	Let $t \ge 1$ be half of an integer and let $H$ be a finite $(AB)^t$-free hypergraph. Then there is a family $\C$ of $(2t-2)$-intersecting $x$-monotone bi-infinite curves in the plane and a finite point set $S$ such that $H(S,\C)$ is isomorphic to $H$.
\end{lemma}

\begin{proof}
	Let $v_1,v_2,\ldots,v_n$ be the vertices of $H$ listed in their $(AB)^{t}$-free order.
	Let $S$ be the points $(i,0)$, for $i=1,2,\ldots,n$.
	Let $e_1,e_2,\ldots,e_m$ be the (finite) set of hyperedges of $H$.
	We draw each edge $e_i$ as a polygonal curve $c_i$ as follows.
	Let $v_j$ be the first vertex that belongs to $e_i$.
	Then $c_i$ follows the horizontal line $y=i$ from minus infinity to $(j-2/3,i)$ where it descends linearly to the point $(j-1/3,-i)$.
	If no such vertex $v_j$ exists, then $c_i$ follows $y=i$ to infinity.
	Otherwise, let $v_k$ be first vertex after $v_j$ that does not belong to $e_i$.
	Then $c_i$ follows the horizontal line $y=-i$ until the point $(k-2/3,-i)$ where it ascends linearly to the point $(k-1/3,i)$.
	If no such vertex $v_k$ exists, then $c_i$ follows $y=-i$ to infinity.
	We proceed drawing $c_i$ in this manner such that the points that correspond to vertices in $e_i$ are above it and the points that correspond to vertices that do not belong to $e_i$ are below it.
	
	Let $\C = \{c_1,c_2,\ldots,c_m\}$. Then $\C$ is a family of $x$-monotone bi-infinite curves such that $H(S,\C)$ is isomorphic to $H$, however, $\C$ might not be $(2t-2)$-intersecting since it might contain `unnecessary' crossings between curves.
	Let $c_i$ and $c_j$ be two curves. 
	Note $c_i$ and $c_j$ may intersect at a finite number of points, where they properly cross each other.
	A \emph{lens} is a maximal connected region bounded from above by $c_i$ and from below by $c_j$ or vice versa.
	A lens is \emph{empty} if it does not contain a point from $S$.
	If $c_i$ and $c_j$ define an empty lens, then we interchange the parts of $c_i$ and $c_j$ that bound this empty lens and redraw $c_i$ and $c_j$ at a small
	neighborhood of their crossing points that are incident to this lens such that they no longer intersect at these points. Thus, we reduce the number of crossing points between $c_i$ and $c_j$ while maintaining the same hypergraph defined by $\C$ and $S$.
	We repeat this process for every empty lens. Since the number of crossing points between the curves is finite, this process terminates and we obtain a family of curves $\C'$ such that $H(S,\C')$ is isomorphic to $H$ and no two curves in $\C'$ define an empty lens.
	
	We claim that $\C'$ is also $(2t-2)$-intersecting.
	Indeed, let $c_A$ and $c_B$ be two curves in $\C'$ that intersect (cross) at $k$ points $p_1,p_2,\ldots,p_k$ and let $A$ and $B$ be the two hyperedges of $H$ that these curve realize.
	Then before $p_1$, after $p_k$ and between any two consecutive crossing points there must be a point from $S$ which is below one curve and above the other curve, for otherwise there would be an empty lens.
	These points correspond to vertices that belong to exactly one of the sets $A$ and $B$ and it follows that there is an $(AB)^{(k+1)/2}$ sequence (assuming without loss of generality that $c_A$ is above $c_B$ before $p_1$).
	Since $H$ is $(AB)^t$-free, it follows that $(k+1)/2<t$, and hence $k \le 2t-2$.
\end{proof}

Theorem \ref{thm:equi} and Lemmas~\ref{lem:parabolas2ABAB} and~\ref{lem:ABAB2parabolas} establish Theorem \ref{thm:equiv}, that is, the equivalence of $ABAB$-free hypergraphs and hypergraphs defined by stabbed pseudo-disks.

\section{Conclusions}
\label{sec:conclusions}

In the paper we have shown that pseudo-disk hypergraphs are equivalent to $ABAB$-free hypergraphs, and they are properly $3$-colorable.
Similar questions can be studied about dual-$ABAB$-free hypergraphs as well, which is equivalent to the so-called cover-decomposition problem for stabbed pseudo-disks.
Another version is to forbid 
$ABABA$-sequences cyclically (instead of linearly); such $3$-uniform hypergraphs have a nice geometric representation, as \emph{convex geometric 3-hypergraphs} without \emph{strongly crossing} edges, see Suk \cite{Suk2013}. 
It is also a natural question to ask whether strongly crossing convex geometric (non-uniform) hypergraphs can be always $3$-colored.

We would also like to remark that having VC-dimension at most $2l-1$ is a weaker assumption than being $(AB)^l$-free. For any $c$ and $m$ there are $m$-uniform hypergraphs of VC-dimension $2$ that are not $c$-colorable; the main construction from both \cite{indec} and \cite{indecconcave} can be generalized from $2$-colors to $c$-colors as $m$-uniform hypergraphs of VC-dimension $2$.

%
%
%

An interesting connection to Radon-partitions is the following. Given three points in $\R^1$, they have a unique Radon-partition into two sets, $A$ and $B$, whose convex hulls intersect; the points must follow each other in the order $A,B,A$, so this cannot happen for $ABA$-free families.
Given four points in $\R^2$, there are two possible Radon-partitions; the first is when three points of $A$ contain the only point of $B$ inside their convex hull, while the second is when there are two points in each of $A$ and $B$ such that their connecting segments intersect.
Note that none of these configurations are possible for points from the symmetric difference of convex pseudo-disks, i.e., if $A$ and $B$ are convex pseudo-disks, then we cannot pick points from $A\setminus B$ and $B\setminus A$ that form a Radon-partition.
We wonder whether this has some higher dimensional generalizations, or is just a coincidence.

The most natural problem left open is whether  $\chi_m(\F_\bigcirc) = 3$ or  $\chi_m(\F_\bigcirc) = 4$.


\bibliographystyle{plainurl}
\bibliography{papers}


\appendix

\section{Proof of Theorem \ref{thm:psdisksimplyconn}}\label{appendix}
We start with some basic topological definitions, without going into too much details.

\begin{defi}
	A region in the plane is \emph{connected} if it is not the union of two (or more) disjoint open sets.
	
	A region in the plane is \emph{path-connected} if any two of its points can be connected by a continuous curve inside the region. 
	
	A region in the plane is \emph{simply connected} if it is path-connected and any Jordan curve inside the region can be continuously contracted into a point inside the region. 
\end{defi}

Therefore, simply connectedness implies path-connectedness, and path-connectedness implies connectedness.
Now we state (a variant of) the Topological Helly Theorem \cite{Helly1930,molnar}.

\begin{lemma}\label{lem:molnartop}
	Let $\F$ be a family of at least three simply connected compact regions such that any three of them intersect and the intersection of any two of the regions is path-connected. Then the intersection of all of the regions in $\F$ is also simply connected.
\end{lemma}

For a complete proof of the base case of Lemma \ref{lem:molnartop} (for three sets), which according to \cite{kare06,tyva05} was not fully proved until recently, see \cite{kare06,tyva05}. In \cite{lenses} it was claimed (referring to \cite{molnar}) that Lemma \ref{lem:molnartop} holds even if we assume that the intersection of any two regions is only connected (instead of path-connected). However, such a stronger variant of Lemma~\ref{lem:molnartop} is not true (not even for three sets) as was shown in \cite{kare06}.
Still, we can avoid some of the topological troubles, as we only need such a statement for pseudo-disks. We will see that the different assumptions in the forthcoming Lemmas \ref{lem:molnarintersection}-\ref{lem:sc} (that we will prove) all hold for stabbed pseudo-disk families, and so do their conclusions.



\begin{lemma}\label{lem:molnarintersection}
	Let $\F$ be a finite family of at least three bounded regions. The intersection of all of the regions in $\F$ is simply connected if all of the following hold:
	\begin{enumerate}
		\item the complement of each region is path-connected;
		\item the intersection of any two of the regions is path-connected; and
		\item if $Q$ is a path-connected component of the intersection of three regions, $q$ is a point in the interior of $Q$ and $r \in \partial Q$, then there exists a path connecting $q$ and $r$ that lies in the interior of $Q$ except for the endpoint $r$.
	\end{enumerate}
\end{lemma}

We point out that Lemma~\ref{lem:molnartop} and Lemma~\ref{lem:molnarintersection} are not implied by one another, because their assumptions do not imply each other.
On one hand, if $\F$ has a region that is not path-connected (say, the disjoint union of two disks), then it will not satisfy the assumptions of Lemma~\ref{lem:molnartop} but it is easy to add to it other regions (say, copies of one of the disks), so that they satisfy Lemma~\ref{lem:molnarintersection}.
On the other hand, there are simply connected compact regions whose complement is not even path-connected (one can easily modify the Warsaw circle to get such regions), so not all of the assumptions of Lemma~\ref{lem:molnarintersection} follow from the assumptions of Lemma~\ref{lem:molnartop}.

To prove Lemma \ref{lem:molnarintersection} we first prove the following dual statement.

\begin{lemma}\label{lem:molnar}
	Let $\F$ be a finite family of at least three bounded regions. The union of all of the regions in $\F$ is simply connected if all of the following hold:
	\begin{enumerate}
		\item each region is path-connected;
		\item the complement of the union of any two of the regions is path-connected;
		\item if $Q$ is a path-connected component of the complement of the union of three regions, $q$ is a point in the interior of $Q$ and $r \in \partial Q$, then there exists a path connecting $q$ and $r$ that lies in the interior of $Q$ except for the endpoint $r$; and
		\item any three of the regions intersect.
	\end{enumerate}
\end{lemma}

Before proving Lemma \ref{lem:molnar} we prove an easy lemma about the connection between path-connectedness and simply connectedness.

\begin{lemma}\label{lem:sc}
	If a region $A$ in the plane and its complement  $\bar A$  are both path-connected, then either $A$ or $\bar A$ is also simply connected. If furthermore $A$ is bounded, then $A$ is simply connected.
\end{lemma}

\begin{proof}
	If $A$ is not simply connected,	then there must be a closed curve $\gamma$ in $A$ which cannot be contracted; such $\gamma$ must necessarily contain a point $p$ of $\bar A$ in its interior. If $\bar A$ contains a point $q$ outside $\gamma$, then every curve connecting $p$ to $q$ must intersect $\gamma$ and thus $\bar A$ cannot be path-connected, a contradiction. Thus $\bar A$ lies completely inside $\gamma$. If $\bar A$ is not simply connected, there must be a closed curve $\gamma'$ in $\bar A$ which contains a point $r$ of $A$ in its interior. As $\gamma'$ is in $\bar A$, it lies completely inside $\gamma$ and thus $\gamma'$ separates $r$ from $\gamma$, contradicting that $A$ is path-connected.
	Therefore $\bar A$ must be simply connected.
	
	Now assume that $A$ is bounded. If $A$ is not simply connected,	then the above defined $\gamma$ would separate any far enough point from $p$, contradicting that $\bar A$ is path-connected.
	Therefore in this case $A$ must be simply connected.
\end{proof}

\begin{remark}
	On the surface of the sphere $\mathbb S^2$ Lemma \ref{lem:sc} translates to an even nicer statement: if a region $A\subset \mathbb S^2$ and its complement $\mathbb S^2\setminus A$ are both path-connected, then both of them are also simply connected.
\end{remark}

\begin{proof}[Proof of Lemma \ref{lem:molnar}]	
	First we prove that the complement of the union of all of the regions in $\F$ is path-connected, for which we do not even use that the regions are bounded.
	
	Let $\F$ be a finite family of at least three (not necessarily bounded) regions that satisfies properties (1)--(4) of Lemma~\ref{lem:molnar}.
	Suppose for contradiction that the complement of the union of all the regions in $\F$ is not path-connected. Consider a containment minimal $\F'\subset \F$ with $|\F'|\ge 2$ for which the complement of $P=\bigcup_{D \in \F'} D$ is not path-connected. As for subfamilies with two regions this cannot hold by our assumptions, we have $|\F'|\ge 3$.
	
	Fix two points, $q$ and $r$, that are in different path-connected components of $\mathbb R^2\setminus P$ (see Figure \ref{fig:k32}).
	Denote the path-connected component of $q$ in $\mathbb R^2\setminus P$ by $Q$.
	The minimality of $\F'$ implies that for any $D\in \F'$ there is a simple curve $\gamma_D$ that connects $q$ to $r$ and avoids $P\setminus\{D\}$. 
	We may assume that each such $\gamma_D$ intersects $\partial Q$, the boundary of $Q$, in exactly one point, $x_D\in \partial D\cap \partial Q$. Indeed, otherwise we can achieve this simply by taking the point $x_D \in \gamma_D \cap \partial Q$ that is farthest from $q$ on $\gamma_D$, and replace the part of $\gamma_D$ between $q$ and $x_D$ by a curve that lies completely inside $Q$ except for $x_D$ (this can be done by the third property of the lemma). 
	
	Denote the part of $\gamma_D$ from $q$ to $x_D$ by $\gamma_D^q$ and the part of $\gamma_D$ from $x_D$ to $r$ by $\gamma_D^r$.
	Due to the construction, $\gamma_D^q\cap \gamma_{D'}^r=\emptyset$ for any $D\ne D'$.
	
	Finally, fix three regions $D_1,D_2,D_3 \in \F'$ and a point $p$ in their intersection. Note that $p$ is different than $r$, $q$ and the points $x_{D_i}$, since $r$ and $q$ belong to the complement of $P$ and each point $x_{D_i}$ belongs to exactly one of the three regions $D_i$, $i=1,2,3$.
	
	Connect $p$ to each point $x_{D_i}$ ($i=1,2,3$) with a curve $\gamma_{D_i}^p$ inside $D_i$. This can be done as $D_i$ is path-connected.
	Note that $\gamma_{D_i}^p\cap \gamma_{D_j}^q=\emptyset$ and $\gamma_{D_i}^p\cap \gamma_{D_j}^r=\emptyset$ for $i \ne j$ as  $\gamma_{D_j}^q$ and $\gamma_{D_j}^r$ are disjoint from $D_i$.
	We have thus obtained a drawing of $K_{3,3}$ in which independent edges do not cross, contradicting the theorems of Hanani-Tutte and Kuratowski.\footnote{As it is \textit{impossibly hard} to draw $K_{3,3}$ where independent edges do not cross, as a second best illustration the reader may consult Figure \ref{fig:k32} where the curves form a drawing of $K_{3,2}$.}
	Thus we conclude that the complement of the union of all the regions is path-connected. 
	
	Now suppose that the regions in $\F$ are bounded. It implies that their union is also bounded. Since their union is also path-connected (as all of them are path-connected), 
	it follows from Lemma~\ref{lem:sc} that the union of all of the regions must also be simply connected, finishing the proof.
\end{proof}

\begin{figure}
	\centering
	\includegraphics[height=4cm]{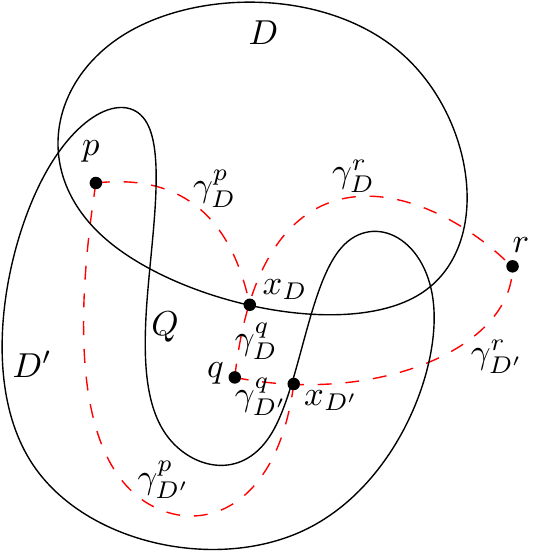}
	\caption{Proof of Lemma \ref{lem:molnar}.}
	\label{fig:k32}	
\end{figure}

The next proof is similar to a proof of a weaker statement in \cite[Lemma~2.9]{lenses}.

\begin{proof}[Proof of Lemma \ref{lem:molnarintersection}] 
	Let $\F$ be a finite family of at least three bounded regions whose complements are path-connected and the intersection of any two of the regions is path-connected and for which for every path-connected component $Q$ of the intersection of three regions and points $q$ in the interior of $Q$ and $r \in \partial Q$ there exists a path connecting $q$ and $r$ that lies in the interior of $Q$ except for the endpoint $r$. We wish to show that the intersection of all of the regions in $\F$ is simply connected.
	
	If the intersection of all of the regions in $\F$ is empty then it is also trivially simply connected. Otherwise there is a point $p$ which is in all the regions.
	
	Project each member of $\F$ onto a sphere such that $p$ is the south pole and the north pole corresponds to infinity. Take the complement (on the sphere) of this projection, and then project it to a plane such that the south pole corresponds to infinity. We denote the family obtained this way from $\F$ by $\F'$. The image of a region $F \in \F$ is denoted by $F'\in \F'$. 
	
	As the complement of each region $F\in \F$ was path-connected, each $F'\in \F'$ is path-connected. Each $F'\in\F'$ contains the point that is the image of the north pole in the second projection. Also, the intersection of any $F_1,F_2\in \F$ is path-connected and thus the union of the complement of any $F_1',F_2' \in \F'$ is also path-connected. Furthermore, the third property of Lemma \ref{lem:molnar} holds for $\F'$ due to the third property of Lemma \ref{lem:molnarintersection}. Finally, any three regions from $\F'$ do intersect as all of them contain the point that corresponded to infinity during the projection step.

	Thus we can apply Lemma~\ref{lem:molnar} to $\F'$ to conclude that the complement of the union of all of the regions in $\F'$ is path-connected. 
	This implies that the intersection $I$ of all the regions in $\F$ is path-connected. Its complement $\bar I$ is also path-connected as it is the union of the complements of the regions in $\F$, where by our assumptions these complements are path-connected and share the point $p$. As $I$ is bounded, Lemma \ref{lem:sc} implies that $I$ is also simply connected.
\end{proof}

Lemmas \ref{lem:molnar} and \ref{lem:molnarintersection} together imply Theorem \ref{thm:psdisksimplyconn}. 

\begin{proof}[Proof of Theorem \ref{thm:psdisksimplyconn}]
	Given a finite family of stabbed pseudo-disks, we only need to show that the properties of Lemmas \ref{lem:molnar} and \ref{lem:molnarintersection} hold. 
	
	First, the regions are indeed bounded. Second, each region is path-connected as it is homeomorphic to a disk. The complement of each region is also path-connected as it is homeomorphic to the closure of $\R^2\setminus D$. The intersection of any two of the regions is a closure of the face of the arrangement consisting of these two regions, thus also homeomorphic to a closed disk and so it is path-connected. The complement of the union of any two of the regions is the outer face of the arrangement of these two pseudo-disks and thus homeomorphic to $\R^2\setminus D$, and thus path-connected. The third property in both lemmas holds also as the regions $Q$ in question are homeomorphic to (open or closed) disks.	
	Finally, any three regions do indeed intersect as it is a stabbed pseudo-disk family.	
\end{proof}

We note that we could also have a proof which relies only on Lemma \ref{lem:molnartop}. We then would need to dualize Lemma \ref{lem:molnartop} and this dual version together with Lemma \ref{lem:molnartop} would also imply Theorem \ref{thm:psdisksimplyconn} the same way as the above proof, using that every face of the arrangement is homeomorphic to a disk and thus fulfills the requirements of Lemma \ref{lem:molnartop} and its (unphrased) dual version.



\end{document}